\documentclass[12pt]{article}
\usepackage{amsthm,amsmath,amssymb,amsfonts}
\usepackage{graphicx}
\usepackage[colorlinks=true,citecolor=black,linkcolor=black,urlcolor=blue]{hyperref}

\newcommand{\arxiv}[1]{\href{http://arxiv.org/abs/#1}{\texttt{arXiv:#1}}}

\theoremstyle{plain}
\newtheorem{theorem}{Theorem}[section]

\newtheorem{corollary}[theorem]{Corollary}

\theoremstyle{definition}
\newtheorem{definition}[theorem]{Definition}
\newtheorem{example}[theorem]{Example}

\theoremstyle{remark}
\newtheorem{remark}[theorem]{Remark}

\title{\bf Integer points enumerator of hypergraphic polytopes}

\author{Marko Pe\v{s}ovi\'c\\
\small Faculty of Civil Engineering, University of Belgrade\\[-0.8ex]
\small\tt mpesovic@grf.bg.ac.rs}

\date{
}

\begin{document}

\maketitle

\begin{abstract}
For a hypergraphic polytope there is a weighted quasisymmetric
function which enumerates positive integer points in its normal
fan and determines its $f-$polynomial. This quasisymmetric
function invariant of hypergraphs extends the Stanley chromatic
symmetric function of simple graphs. We consider a certain
combinatorial Hopf algebra of hypergraphs and show that universal
morphism to quasisymmetric functions coincides with this
enumerator function. We calculate the $f-$polynomial of uniform
hypergraphic polytopes.\\

\end{abstract}

\section{Introduction }

The theory of combinatorial Hopf algebras developed by Aguiar,
Bergeron and Sottile in the seminal paper \cite{ABS} provides an
algebraic framework for symmetric and quasisymmetric generating
functions arising in enumerative combinatorics. Extensive studies
of various combinatorial Hopf algebras are initiated recently
\cite{BHM},\cite{BJR},\cite{GS},\cite{GSJ}. The geometric
interpretation of the corresponding (quasi)symmetric functions was
first given for matroids \cite{BJR} and then for simple graphs
\cite{G1} and building sets \cite{G2}. The quasisymmetric function
invariants are expressed as integer points enumerators associated
to generalized permutohedra. This class of polytopes introduced by
Postnikov \cite{P} is distinguished with rich combinatorial
structure. The comprehensive treatment of weighted integer points
enumerators associated to generalized permutohedra is carried out
by Gruji\'c et al. \cite{GPS}. In this paper we consider a certain
naturally defined non-cocommutative combinatorial Hopf algebra of
hypergraphs and show that the derived quasisymmetric function
invariant of hypergraphs is integer points enumerator of
hypergraphic polytopes (Theorem \ref{main}).

\section{Combinatorial Hopf algebra of hypergraphs $\mathcal{HG}$}

A \emph{combinatorial Hopf algebra} is a pair
$(\mathcal{H},\zeta)$ of a graded connected Hopf algebra
$\mathcal{H}=\oplus_{n\geq0}\mathcal{H}_n$ over a field
$\bold{k}$, whose homogeneous components $\mathcal{H}_n,n\geq0$
are finite-dimensional, and a multiplicative linear functional
$\zeta:\mathcal{H}\rightarrow\bold{k}$ called \emph{character}. We
consider a combinatorial Hopf algebra structure on hypergraphs
different from the chromatic Hopf algebra of hypergraphs studied
in \cite{GSJ}. The difference is in the coalgebra structures based
on different combinatorial constructions, which is manifested in
(non)co-commutativity. It extends the Hopf algebra of building set
studied by Gruji\'c in \cite{G1}. This Hopf algebra of hypergraphs
can be derived from the Hopf monoid structure on hypergraphs
introduced in \cite{AA}.

A \emph{hypergraph} $\mathbf{H}$ on the vertex set $V$ is a
collection of nonempty subsets $H\subseteq V$, called
\emph{hyperedges}. We assume that there are no ghost vertices,
i.e. $\mathbf{H}$ contains all singletons $\{i\}, i\in V$. A
hypergraph $\mathbf{H}$ is \emph{connected} if it can not be
represented as a disjoint union of hypergraphs
$\mathbf{H}_1\sqcup\mathbf{H}_2$. Every hypergraph $\mathbf{H}$
splits into its connected components. Let $c(\bold{H})$ be the
number of connected components of $\bold{H}$. Hypergraphs
$\mathbf{H}_1$ and $\mathbf{H}_2$ are isomorphic if there is a
bijection of their sets of vertices that sends hyperedges to
hyperedges. Let $\mathcal{HG}=\bigoplus_{n\geq0}\mathcal{HG}_n,$
where $\mathcal{HG}_n$ is the linear span of isomorphism classes
$[\mathbf{H}]$ of hypergraphs on the set $[n]$.

\begin{definition}
For a subset $S\subseteq[n]$ the \emph{restriction} $\bold{H}|_S$
and the \emph{contraction} $\bold{H}/S$ are defined by
$$\mathbf{H}|_S=\{H\in\mathbf{H}\,\,:\,\,H\subseteq S\},$$
$$\mathbf{H}/S=\{H\setminus S\,\,:\,\,H\in\mathbf{H}\}.$$
\end{definition}

\noindent Define a \emph{product} and a \emph{coproduct} on the
linear space $\mathcal{HG}$ by
$$[\mathbf{H}_1]\cdot[\mathbf{H}_2]\;=\;[\mathbf{H}_1\sqcup\mathbf{H}_2],$$
$$\Delta([\mathbf{H}])=\sum_{S\subset [n]}[\mathbf{H}|_S]\otimes[\mathbf{H}/S].$$
The straightforward checking shows that the space $\mathcal{HG}$
with the above operations and the unit
$\eta:\bold{k}\rightarrow\mathcal{HG}$ given by
$\eta(1)=[\mathbf{H}_\emptyset]$ (the empty hypergraph) and the
counit $\epsilon:\mathcal{HG}\rightarrow\bold{k}$ which is the
projection on the component $\mathcal{HG}_0=\bold{k}$, become a
graded connected commutative and non-cocommutative bialgebra.
Since graded connected bialgebras of finite type posses antipodes,
$\mathcal{HG}$ is in fact a Hopf algebra. The formula for antipode
$S:\mathcal{HG}\rightarrow\mathcal{HG}$ is derived from the
general Takeuchi's formula \cite{T}
$$S([\bold{H}])=\sum_{k\geq
1}(-1)^k\sum_{\mathcal{L}_k}\prod_{j=1}^{k}([\bold{H}]\mid_{I_j})/I_{j-1},$$
where the inner sum goes over all chains of subsets
$\mathcal{L}_k:\emptyset=I_0\subset I_1\subset\cdots\subset
I_{k-1}\subset I_k=V$. Define a character
$\zeta:\mathcal{HG}\rightarrow\mathbf{k}$ by $\zeta([\bold{H}])=1$
if $\bold{H}$ is discrete, i.e. contains only singletons and
$\zeta([\bold{H}])=0$ otherwise. This determines the combinatorial
Hopf algebra $(\mathcal{HG},\zeta)$.

\section{Integer points enumerator}

In this section we review the definition of the integer points
enumerator of a generalized permutohedron introduced in
\cite{GPS}.

For a point $(a_1,a_2,\ldots,a_n)\in\mathbb{R}^n$ with increasing coordinates $a_1<\cdots<a_n$ let us define the set
$\Omega(a_1,a_2,\ldots,a_n)$ by
$$\Omega(a_1,a_2,\ldots,a_n)=\{(a_{\omega(1)},a_{\omega(2)},\ldots,a_{\omega(n)})\;:\;\omega\in\mathfrak{S}_n\},$$
where $\mathfrak{S}_n$ is the permutation group of the set $[n]$.
The convex hull of the set $\Omega(a_1,a_2,\ldots,a_n)$ is a
\emph{standard $(n-1)-$dimensional permutohedron $Pe^{n-1}$}. The
$d-$dimensional faces of $Pe^{n-1}$ are in one-to-one
correspondence with set compositions
$\mathcal{C}=C_1|C_2|\cdots|C_{n-d}$ of the set $[n],$ see
\cite{P}, Proposition 2.6. By this correspondence and the obvious
correspondence between set compositions and flags of subsets we
identify faces of $Pe^{n-1}$ with flags
$\mathcal{F}:\emptyset=F_0\subset F_1\subset\cdots\subset
F_{n-d}=[n]$. The dimension of a face and length of the
corresponding flag is related by
$\dim(\mathcal{F})=n-|\mathcal{F}|$.

The normal fan $\mathcal{N}(Pe^{n-1})$ of the standard
permutohedron is the braid arrangement fan $\{x_i=x_j\}_{1\leq
i<j\leq n}$ in the space $\mathbb{R}^n$. The dimension of the
normal cone $C_\mathcal{F}$ at the face $\mathcal{F}$ is
$\dim(C_\mathcal{F})\;=\;|\mathcal{F}|.$ The relative interior
points $\omega\in C^\circ_\mathcal{F}$ are characterized by the
condition that their coordinates are constant on $F_i\setminus
F_{i-1}$ and increase $\omega|_{F_i\setminus
F_{i-1}}<\omega|_{F_{i+1}\setminus F_i}$ with $i$. A positive
integer vector $\omega\in\mathbb{Z}^n_+$ belongs to
$C^\circ_\mathcal{F}$ if the weight function
$\omega^\ast(x)\;=\;\langle\omega, x\rangle$ is maximized on
$Pe^{n-1}$ along a face $\mathcal{F}$.

\begin{definition}
For a flag $\mathcal{F}$ let $M_\mathcal{F}$ be the enumerator of
interior positive integer points $\omega\in\mathbb{Z}^n_+$ of the
corresponding cone $C_\mathcal{F}$
$$
M_\mathcal{F}\;\;=\sum_{\omega\in\mathbb{Z}^n_+\,\cap\,C^\circ_\mathcal{F}}\bold{x}_\omega,$$
where $\bold{x}_\omega=x_{\omega_1}x_{\omega_2}\cdots
x_{\omega_n}$.
\end{definition}
\noindent The enumerator $M_\mathcal{F}$ is a monomial
quasisymmetric function depending only of the composition
$\mathrm{type}(\mathcal{F})=(|F_1|,|F_2\setminus
F_1|,\ldots,|F_k\setminus F_{k-1}|)$.

The fan $\mathcal{N}$ is a \emph{coaresement} of
$\mathcal{N}(Pe^{n-1})$ if every cone in $\mathcal{N}$ is a union
of cones of $\mathcal{N}(Pe^{n-1})$. An $(n-1)-$dimensional
\emph{generalized permutohedron} $Q$ is a convex polytope whose
normal fan $\mathcal{N}(Q)$ is a coaresement of
$\mathcal{N}(Pe^{n-1})$. There is a map
$\pi_Q:L(Pe^{n-1})\rightarrow L(Q)$ between face lattices given by
$$\pi_Q(\mathcal{F})=G
\;\;\;\text{ if and only if }\;\;\; C^\circ_\mathcal{F}\subseteq
C_G^\circ,$$ where $C_G^\circ$ is the relative interior of the
normal cone $C_G$ at the face $G\in L(Q).$

\begin{definition}\label{general}
For an $(n-1)-$generalized permutohedron $Q$ let $F_q(Q)$ be the
weighted integer points enumerator
$$
F_q(Q)\;\;=\sum_{\omega\in\mathbb{Z}^n_+}q^{\dim(\pi_Q(\mathcal{F}_\omega))}\bold{x}_\omega\;\;=\sum_{\mathcal{F}\in
L(Pe^{n-1})}q^{\dim(\pi_Q(\mathcal{F}))}M_\mathcal{F},
$$
where $\mathcal{F}_\omega$ is a unique face of $Pe^{n-1}$
containing $\omega$ in the relative interior.
\end{definition}

\begin{remark}\label{remark}
It is shown in \cite[Theorem 4.4]{GPS} that the enumerator
$F_q(Q)$ contains the information about the $f$-vector of a
generalized permutohedron $Q$. More precisely, the principal
specialization of $F_q(Q)$ gives the $f$-polynomial of $Q$

\begin{equation}\label{face}
f(Q,q)=(-1)^{n}\mathbf{ps}(F_{-q}(Q))(-1).
\end{equation}
\noindent Recall that the principal specialization
$\mathbf{ps}(F)(m)$ of a quasisymmetric function $F$ in variables
$x_1, x_2,\ldots$ is a polynomial in $m$ obtained from the
evaluation map at $x_i=1, i=1,\ldots,m$ and $x_i=0$ for $i>m$.
\end{remark}

\section{The hypergraphic polytope}

For the standard basis vectors $e_i,1\leq i\leq n$ in
$\mathbb{R}^n$ let $\Delta_H=\mathrm{conv}\{e_i\,:\;i\in H\}$ be
the simplex determined by a subset $H\subset[n]$. The
\emph{hypergraphic polytope} of a hypergraph $\mathbf{H}$ on $[n]$
is the Minkowski sum of simplices
$$P_{\,\mathbf{H}}=\sum_{H\in\,\mathbf{H}}\Delta_H.$$
As generalized permutohedra can be described as the Minkowski sum
of delated simplices (see \cite{P}) we have that hypergraphic
polytopes are generalized permutohedra. For the following
description of $P_\mathbf{H}$ see \cite[Section 1.5]{BP} and the
references within it. Let
$\mathbf{H}=\mathbf{H}_1\sqcup\mathbf{H}_2\sqcup\cdots\sqcup\mathbf{H}_k$
be the decomposition into connected components. Then
$P_{\,\mathbf{H}}=P_{\,\mathbf{H}_1}\times
P_{\,\mathbf{H}_2}\times\cdots\times P_{\,\mathbf{H}_k}$ and
$\dim(P_{\,\mathbf{H}})=n-k.$ For connected hypergraphs
$P_{\,\mathbf{H}}$ can be described as the intersection of the
hyperplane
$H_{\,\mathbf{H}}:=\left\{x\in\mathbb{R}^n\;\;:\;\;\sum_{i=1}^nx_i=|\mathbf{H}|\right\}$
with the halfspaces
$H_{S,\,\geq}:=\left\{x\in\mathbb{R}^n\;\;:\;\;\sum_{i\in
S}x_i\geq|\mathbf{H}|_S|\right\}$ corresponding to all proper
subsets $S\subset[n]$. It follows that $P_{\,\mathbf{H}}$ can be
obtained by iteratively cutting the standard simplex
$\Delta_{[n]}$ by the hyperplanes $H_{S,\,\geq}$ corresponding to
proper subsets $S$. For instance the standard permutohedron
$Pe^{n-1}$ is a hypergraphic polytope $P_{\mathbf{C}_n}$
corresponding to the complete hypergraph $\mathbf{C}_n$ consisting
of all subsets of $[n]$.

\begin{definition}\label{rank}
For a connected hypergraph $\mathbf{H}$ the
\emph{$\mathbf{H}-$rank} is a map
$\mathrm{rk}_{\,\mathbf{H}}:L(Pe^{n-1})\rightarrow\{0,1,\ldots,n-1\}$
given by
$$\mathrm{rk}_{\mathbf{H}}(\mathcal{F})\,=\,\dim(\pi_{P_\mathbf{H}}(\mathcal{F})).$$
\end{definition}

\noindent Subsequently we deal only with connected hypergraphs.
The quasisymmetric function $F_q(P_{\mathbf{H}})$ corresponding to
a hypergraphic polytope $P_\mathbf{H}$, according to Definitions
\ref{general} and \ref{rank}, depends only on the rank function

\begin{equation}\label{F_q}
F_q(P_{\,\mathbf{H}}\,)\;\;=\sum_{\mathcal{F}\in
L(Pe^{n-1})}q^{\mathrm{rk}_{\mathbf{H}}(\mathcal{F})}M_\mathcal{F}.
\end{equation}

\noindent We extend the ground field $\bold{k}$ to the field of
rational function $\bold{k}(q)$ in a variable $q$ and consider the
Hopf algebra $\mathcal{HG}$ over this extended field. Let
$\mathrm{rk}(\mathbf{H})=n-c(\mathbf{H})$ for hypergraphs on $n$
vertices. Define a linear functional
$\zeta_q:\mathcal{HG}\rightarrow\mathbf{k}(q)$ with
$$\zeta_q([\mathbf{H}])=q^{\mathrm{rk}(\mathbf{H})}=q^{n-c(\mathbf{H})},$$
which is obviously multiplicative. By the characterization of the
combinatorial Hopf algebra of quasisymmetric functions $(QSym,
\zeta_Q)$ as a terminal object (\cite[Theorem 4.1]{ABS}) there
exists a unique morphism of combinatorial Hopf alegbras
$\Psi_q\;:\;(\mathcal{H},\;\zeta_q)\;\rightarrow\;(\mathcal{Q}Sym,\;\zeta_\mathcal{Q})$
given on monomial basis by
$$\Psi_q([\mathbf{H}])=\sum_{\alpha\models n}(\zeta_q)_\alpha([\mathbf{H}])M_\alpha.$$

\noindent We determine the coefficients by monomial functions in
the above expansion more explicitly. For a hypergraph $\mathbf{H}$
define its splitting hypergraph $\textbf{H}/\mathcal{F}$ by a flag
$\mathcal{F}$ with
$$\mathbf{H}/\mathcal{F}\;=\;\bigsqcup_{i=1}^k\mathbf{H}|_{F_i}/F_{i-1}.$$ The coefficient corresponding to a
composition $\alpha=(\alpha_1,\alpha_2,\ldots,\alpha_k)\models n$
is a polynomial in $q$ determined by
$$(\zeta_q)_\alpha([\mathbf{H}])=\sum_{\mathcal{F}\,:\,\mathrm{type}(\mathcal{F})=\alpha}\prod_{i=1}^kq^{\mathrm{rk}(\mathbf{H}|_{F_i}/F_{i-1})}=
\sum_{\mathcal{F}\,:\,\mathrm{type}(\mathcal{F})=\alpha}q^{\mathrm{rk}(\mathbf{H}/\mathcal{F})},$$
where the sum is over all flags $\mathcal{F}:\emptyset =:
F_0\subset F_1 \subset\cdots\subset F_k := [n]$ of the type
$\alpha$ and

\begin{equation}\label{splitt}
\mathrm{rk}(\mathbf{H}/\mathcal{F})=\sum_{i=1}^k\mathrm{rk}(\mathbf{H}|_{F_i}/F_{i-1})=n-\sum_{i=1}^kc(\mathbf{H}|_{F_i}/F_{i-1}).
\end{equation} By this correspondence, we have
\begin{equation}\label{Psi}
\Psi_q([\mathbf{H}])\;\;\;=\sum_{\mathcal{F}\in L(Pe^{n-1})}q^{\mathrm{rk}(\mathbf{H}/\mathcal{F})}M_{\mathrm{type}(\mathcal{F})}.
\end{equation}

Now we have two quasisymmetric functions associated to hypergraphs
whose expansions in monomial bases are given by $(\ref{F_q})$ and
$(\ref{Psi})$. We show that they actually coincide which describes
the corresponding hypergraphic quasisymmetric invariant
algebraically and geometrically.

\begin{theorem}\label{main}
For a connected hypergraph $\mathbf{H}$ the integer points
enumerator $F_q(P_{\,\mathbf{H}})$ associated to a hypergraphic
polytope and the quasisymmetric function $\Psi_q([\mathbf{H}])$
coincide
$$F_q(P_{\,\mathbf{H}}) = \Psi_q([\mathbf{H}]).$$
\end{theorem}
\begin{proof}
Let $\mathbf{H}$ be a connected hypergraph on the set $[n]$ and
$\mathcal{F}:\emptyset=F_0\subset F_1\subset
F_2\subset\cdots\subset F_m=[n]$ be a flag of subsets of $[n]$. It
is sufficient to prove that
\begin{equation}\label{jedan}
\mathrm{rk}_\mathbf{H}(\mathcal{F})=\mathrm{rk}(\mathbf{H}/\mathcal{F}).
\end{equation}
For this we need to determine the face $G$ of the hypergraphic
polytope $P_\mathbf{H}$ along which the weight function
$\omega^{\ast}$ is maximized for an arbitrary $\omega\in
C_\mathcal{F}^{\circ}$. Since $P_\mathbf{H}$ is the Minkowski sum
of simplices $\Delta_H$ for $H\in\mathbf{H}$ the face $G$ is
itself a Minkowski sum of the form
$G=\sum_{H\in\mathbf{H}}(\Delta_H)_\mathcal{F}$ where
$(\Delta_H)_\mathcal{F}$ is a unique face of $\Delta_H$ along
which the weight function $\omega^{\ast}$ is maximized for
$\omega\in C_\mathcal{F}^{\circ}$. Let
$\omega=(\omega_1,\ldots,\omega_n)$ where $\omega_i=j$ if $i\in
F_j\setminus F_{j-1}$ for $i=1,\ldots,n$. Then $\omega\in
C_\mathcal{F}^{\circ}$ and we can convince that
$(\Delta_H)_\mathcal{F}=\Delta_{H\setminus F_{j-1}}$ where
$j=\min\{k\mid H\subset F_k\}$. Denote by $\mathbf{H}_j$ the
collection of all $H\in\mathbf{H}$ with $j=\min\{k\mid H\subset
F_k\}$ for $j=1,\ldots,m$. We can represent the face $G$ as
$G=\sum_{j=1}^{m}\sum_{H\in\mathbf{H}_j}\Delta_{H\setminus
F_{j-1}},$ which shows that $G$ is precisely a hypergraphic
polytope corresponding to the splitting hypergraph

$$G=P_{\mathbf{H}/\mathcal{F}}.$$ The equation $(\ref{jedan})$ follows from the
fact that $\dim
P_{\mathbf{H}/\mathcal{F}}=\mathrm{rk}(\mathbf{H}/\mathcal{F}),$
which is given by $(\ref{splitt})$.

\end{proof}
\noindent As a corollary, by Remark \ref{remark} and the equation
$(\ref{face})$ within it, we can derive the $f$-polynomial of a
hypergraphic polytope $P_\mathbf{H}$ in a purely algebraic way.

\begin{corollary}\label{cor}
The $f-$polynomial of a hypergraphic polytope $P_{\,\mathbf{H}}$
is determined by the principal specialization
$$f(P_{\,\mathbf{H}},q)=(-1)^n\,\mathbf{ps}\left(\Psi_{-q}([\mathbf{H}])\right)(-1).$$
\end{corollary}

We proceed with some examples and calculations.

\begin{example}
Let $\bold{U}_{n,k}$ be the $k$-uniform hypergraph containing all
$k$-elements subsets of $[n]$ with $k>1$. Divide flags into two
families depending on whether they contain a $k$-elements subset.
Let $\circ$ be a bilinear operation on quasisymmetric functions
given on the monomial bases by concatenation $M_\alpha\circ
M_\beta=M_{\alpha\cdot\beta}.$ The flags that contain $k$-elements
subset contribute to $\Psi([\bold{U}_{n,k}])$ with
$$\sum_{i=1}^{k}{n\choose k-i,i,n-k}q^{i-1}M_{(1)}^{k-i}\circ M_{(i)}
\circ\Psi_q([\mathbf{C}_{n-k}]).$$ The contribution to
$\Psi([\bold{U}_{n,k}])$ of the remaining flags is
$$\sum_{0\leq a<k<n-b\leq n}{n\choose a,b,n-a-b}q^{n-a-b-1}M_{(1)}^a\circ M_{(n-a-b)}
\circ\Psi_q(\mathbf{C}_b).$$ By Corollary \ref{cor} since the
principal specialization respects the operation $\circ$ it follows
from $P_{\mathbf{C}_m}=Pe^{m-1}$ that

$$f(P_{\,\bold{U}_{n,k}},q)=\sum_{i=1}^{k}{n\choose k-i,i,n-k}q^{i-1}f(Pe^{n-k-1},q)+$$
$$+\sum_{0\leq a<k<n-b\leq n}{n\choose
a,b,n-a-b}q^{n-a-b-1}f(Pe^{b-1},q).$$

\end{example}

\begin{example}
The hypergraphic polytope $PS^{n-1}$ corresponding to the
hypergraph $\{[1],[2],\ldots,[n]\}$ is known as the Pitman-Stanley
polytope. It is combinatorially equivalent to the $(n-1)$-cube
\cite[Proposition 8.10]{P}. The following recursion is satisfied

$$F_q(PS^{n})=F_q(PS^{n-1})M_{(1)}+(q-1)(F_q(PS^{n-1}))_{+1},$$
where $_{+1}$ is given on monomial bases by
$(M_{(i_1,i_2,\ldots,i_k)})_{+1}=M_{(i_1,i_2\ldots,\i_k+1)}$. It
can be seen by dividing flags into two families according to the
position of the element $n$. To a flag
$\mathcal{F}:\emptyset=F_0\subset F_1\subset\cdots\subset F_m=[n]$
we associate the flag
$\widetilde{\mathcal{F}}:\emptyset=F_0\subset
F_1\setminus\{n\}\subset\cdots\subset F_m\setminus\{n\}=[n-1]$. If
$n\in F_k$ for some $k<m$ then
$\mathrm{rk}_{PS^{n}}(\mathcal{F})=\mathrm{rk}_{PS^{n-1}}(\mathcal{F})$
and if $n\notin F_k$ for $k<m$ then
$\mathrm{rk}_{PS^{n}}(\mathcal{F})=\mathrm{rk}_{PS^{n-1}}(\mathcal{F})+1$.
The principal specialization of the previous recursion formula
gives
$$f_q(PS_{n})=(2+q)f_q(PS_{n-1}),$$
consequently $f_q(PS_{n})=(2+q)^n$ which reflects the fact that
$PS^{n}$ is an $n$-cube.

\end{example}

\begin{example}
If $\Gamma$ is a simple graph, the corresponding hypergraphic
polytope $P_\Gamma$ is the graphic zonotope
$$P_\Gamma=\sum_{\{i,j\}\in\Gamma}\Delta_{e_i,e_j}.$$
Simple graphs generate the Hopf subalgebra of $\mathcal{HG}$ which
is isomorphic to the chromatic Hopf algebra of graphs. Therefore
$F_q(P_\Gamma)$ is the $q$-analogue of the Stanley chromatic
symmetric function of graphs introduced in \cite{G1}.
\end{example}

\begin{example}
Simplicial complexes generate another Hopf subalgebra of
$\mathcal{HG}$ which is isomorphic to the Hopf algebra of
simplicial complexes introduced in \cite{GSJ} and studied more
extensively in \cite{BHM}. It is shown in \cite[Lemma 21.2]{AA}
that hypergraphic polytopes $P_{K}$ and $P_{K^{1}}$ corresponding
to a simplicial complex $K$ and its $1$-skeleton $K^{1}$ are
normally equivalent and therefore have the same enumerators

$$F_q(P_K)=F_q(P_{K^{1}}).$$

\end{example}

\end{document}